\documentclass[a4paper,10pt]{article}
\usepackage[english]{babel}
\usepackage[utf8]{inputenc}
\usepackage[left=3.1cm, right=3.1cm, top=4cm, bottom=4cm]{geometry}

\usepackage[nottoc]{tocbibind}
\usepackage{amsthm}
\usepackage{bm}
\usepackage{hyperref}
\usepackage{enumerate}
\usepackage{graphicx, psfrag, amscd, amssymb,amsmath,amsthm, amsfonts}
\usepackage{mathrsfs}
\usepackage{cite}
\usepackage{titling}
\predate{}
\postdate{}

\newtheorem{theorem}{Theorem}[section]           % theorems
\newtheorem{corollary}[theorem]{Corollary}       % corollaries
\newtheorem{lemma}[theorem]{Lemma}               % lemmas
\theoremstyle{definition} 
\newtheorem{definition}[theorem]{Definition}     % definitions
\newtheorem{example}[theorem]{Example}           % examples
\theoremstyle{remark}
\newtheorem{remark}[theorem]{Remark}             % remarks

\title{On the Refinement of Certain Statistics on Alternating Words}
\author{Chia-An Hsu, Hsu-Lin Chien, Han-Chun Chan,\\ Bin-Shun Sun, and Yuan-Ting Huang}
\date{}
\begin{document}
\maketitle

\begin{abstract}
    In this paper, we investigate statistics on alternating words under correspondence between ``possible reflection paths within several layers of glass'' and ``alternating words''.  For $v=(v_1,v_2,\cdots,v_n)\in\mathbb{Z}^{n}$, we say $P$ is a path within $n$ glass plates corresponding to $v$, if $P$ has exactly $v_i$ reflections occurring at the $i^{\rm{th}}$ plate for all $i\in\{1,2,\cdots,n\}$.
    We give a recursion for the number of paths corresponding to $v$ satisfying $v \in \mathbb{Z}^n$ and $\sum_{i\geq 1} v_i=m$. Also, we establish recursions for statistics around the number of paths corresponding to a given vector $v\in\mathbb{Z}^n$ and a closed form for $n=3$.
    Finally, we give a equivalent condition for the existence of path corresponding to a given vector $v$.
    %We gave some generating functions and recursions of this combinatorial statistics, and gave a closed form of this statistics for the case of $3$ plates. We also gave a sufficient and necessary condition to discriminate whether there is a path corresponding to a given three dimensional vector or not. 
\end{abstract}

%-------------------------------------------------------------------------------------------------------------------------------------------------------------------------------------------------------------------------
\section{Introduction}
 There are several results about reflections across glass plates. For example, if $a_m$ is the number of distinct paths within $3$ glass plates, such that there are exactly $m$ reflections, then the recursion of sequence $\{a_m\}_{m\in\mathbb{N}}$ is $a_m=a_{m-1}+a_{m-2}$. In \cite{junge1973polynomials}, Junge and Hoggatt generalized the statistic $a_n$ to $m$ plates, and study a related matrix equation and its characteristic polynomial. In \cite{hoggatt1979reflections}, Hoggatt considered statistics in paths of fixed length. In present paper, we will focus on other statistics which are also quite natural.
\begin{definition}\label{def}
Within $n$ glass plates: 
\begin{enumerate}
    \item We say $P$ is a \emph{path between $n$ glass plates} if $P$ is a polygonal curve with peaks and valleys lying in some plates. In this case, peaks and valleys are called reflections.
    \item For $i<n$, we say a path $P$ \emph{starts at $i^{\rm{th}}$ plate} if the first reflection occurs at $j^{\rm{th}}$ plate such that $i<j$. 
    \item We say $P$ is a path \emph{corresponding to a given vector $v=(v_1,v_2,\cdots,v_n)\in\mathbb{Z}^n$} if $P$ starts at first plate and has exactly $v_i$ reflections at the $i^{\rm{th}}$ plate for all $i\in\{1,2,\cdots,n\}$. On the other hand, we call $v(P)$ the vector determined by path $P$.
    \item For $v \in \mathbb{Z}^n$, $N(v)$ is the number of paths  $P$ such that $v(P)=v$. 
    \item $a^n_m$ is the number of paths $P$ such that $\sum v(P)=m$.
    \item $b^n_m$ is the number of $v$ satisfying $v=v(P)$ and $\sum v(P)=m$ for some path $P$.
    %distinct vectors which is corresponding to some paths which start at the first plate and have exactly $m$ reflections.
\end{enumerate}
\end{definition}

\begin{example}\label{exa:a^3_4}
The first two rows of figure \ref{fig:a^3_4} show the distinct $a^3_4=8$ paths and their corresponding vectors within $3$ glass plates with $4$ reflections. Note that the vector $(2,1,1)$ is corresponding to $2$ different paths, so we have $N(2,1,1)=2$. Moreover, since the $8$ paths corresponding to only $6$ different vectors, we get $b^3_4=6$.
\end{example}

\begin{figure}[htp]
    \centering
    \includegraphics[width=14cm]{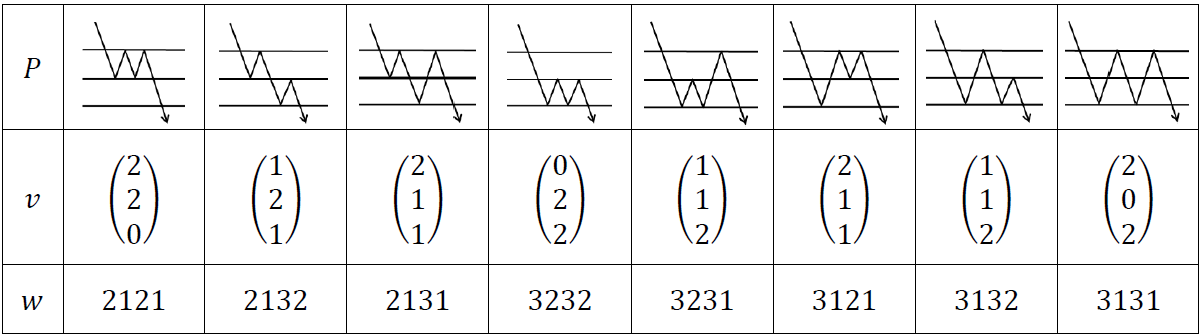}
    \caption{Paths within $3$ glass plates with $4$ reflections}
    \label{fig:a^3_4}
\end{figure}

A word $w=w_1 w_2 \cdots w_n$ is (down-up) alternating if $w_1>w_2<w_3>w_4<\cdots$. If $N_{n,m}$ is the number of alternating words of length $m$ over alphabet $\{1,2,\cdots,n\}$, then $N_{n,m}=a^n_m$ for all $m\ge2$, $n\ge 2$ because of one-to-one correspondence between paths within plates and alternating words (see figure \ref{fig:a^3_4}). Moreover, in view of alternating words, $b^n_m$ is the number of equivalence classes where $w_1$ and $w_2$ are related if the numbers of each alphabet in $w_1$ and $w_2$ are the same respectively. In figure \ref{fig:a^3_4}, we see $2131$ and $3121$ are equivalent since they both contain two ``$1$'', one  ``$2$'' and one ``$3$''. $N(v)$ is the number of words with $i$ occurring $v_i$ times. See \cite{gao2016pattern} for more information about alternating words. 

%-------------------------------------------------------------------------------------------------------------------------------------------------------------------------------------------------------------------------
\section{A closed form of $a^n_m$}
\begin{definition}
We define $a^n_{m,j}$ as the number of paths within $n$ glass plates which start at the first plate with $m$ reflections, and ``the last reflection'' occur at the $j^{\rm{th}}$ plate.  
\end{definition}

From definition, we know that $a^n_m=\sum_{j=1}^n a^n_{m,j}$. Moreover, it is easy to see $a^n_{2i,j}=\sum_{j<j'} a^n_{2i-1,j'}$ and $a^n_{2i+1,j}=\sum_{j'<j} a^n_{2i,j'}$ which lead to the following theorem.

\begin{theorem}
For positive integers $n\ge 2$, $m\ge 1$, and $1\le j\le n$, we have
\[a^n_{m,j}=\sum_{\ell=0}^{\lfloor{\frac{m-1}{2}}\rfloor}(-1)^{\ell+{\lfloor{\frac{m-1}{2}}\rfloor}}a_{\ell} C(n-j+\lfloor{\frac{m-1}{2}}\rfloor-\ell,m-1-2\ell)\]
where $a_{\ell}=\sum_{i=1}^{\ell}(-1)^{i+1}C(n+i-1,2i)\times a_{\ell-i}$ with $a_0=1$, and $C(\alpha,\beta)={\alpha !}/[\beta !(\alpha-\beta)!]$.
\end{theorem}

\begin{proof}
We give a proof by induction on $m$. It is easy to see that $a^n_{1,j}=C(n-j,0)$ and $a^n_{2,j}=C(n-j,1)$ which satisfy the formula. If $m$ is even, say $m=2r$, then by induction hypothesis and the recursion mentioned above, we have
\begin{equation}\nonumber
    \begin{aligned}
    a^n_{2r,j}&=\sum_{j'=j+1}^n a^n_{2r-1,j'}=\sum_{j'=j+1}^n\sum_{\ell=0}^{r-1}(-1)^{\ell+r-1}a_{\ell} C(n-j'+r-1-\ell,2r-2-2\ell)\\
    &=\sum_{\ell=0}^{r-1}(-1)^{\ell+r-1}a_{\ell} \sum_{j'=j+1}^n C(n-j'+r-1-\ell,2r-2-2\ell)\\
    &=\sum_{\ell=0}^{r-1}(-1)^{\ell+r-1}a_{\ell} C(n-j+r-1-\ell,2r-1-2\ell).
    \end{aligned}
    \end{equation}
Similarly, if $m$ is odd, say $m=2r+1$, we have    
\begin{equation}\nonumber
    \begin{aligned}
    a^n_{2r+1,j}&=\sum_{\ell=0}^{r-1}(-1)^{\ell+r-1}a_{\ell} \sum_{j'=1}^{j-1} C(n-j'+r-1-\ell,2r-1-2\ell)\\
    &=\sum_{\ell=0}^{r-1}(-1)^{\ell+r-1}a_{\ell} [C(n+r-\ell-1,2r-2\ell)-C(n-j+r-\ell,2r-2\ell)]\\
    &=\sum_{\ell=0}^{r-1}(-1)^{\ell+r}a_{\ell} C(n-j+r-\ell,2r-2\ell)+\sum_{\ell=0}^{r-1}(-1)^{\ell+r-1}a_{\ell} C(n+r-\ell-1,2r-2\ell)\\
    &=\sum_{\ell=0}^{r-1}(-1)^{\ell+r}a_{\ell} C(n-j+r-\ell,2r-2\ell)+a_r\\
    &=\sum_{\ell=0}^{r}(-1)^{\ell+r}a_{\ell} C(n-j+r-\ell,2r-2\ell).
    \end{aligned}
    \end{equation}    
\end{proof}

\section{The generating function and recursions of $N(v)$}
In this section, our goal is that for $v=(v_1,v_2,\cdots,v_n)\in\mathbb{Z}^n$, we want to find the generating function $D$ of $N(v)$ such that
\[D(t_1,t_2,\cdots,t_n)=\sum_{(v_1,v_2,\cdots,v_n)\in\mathbb{Z}^n}N(v_1,v_2,\cdots,v_n)\cdot t_1^{v_1}t_2^{v_2}\cdots t_n^{v_n}.\]
For this purpose, we have the following definitions.
\begin{definition}
We define $D^n_i$ to be the generating function corresponding to the number of possible paths within $n$ plates starting at the $i^{\rm{th}}$ plate.
\end{definition}

\begin{definition}
For every function of form $D(t_1,t_2,\cdots,t_n)$, we define \[D'(t_1,t_2,\cdots,t_n)=D(t_n,t_{n-1},\cdots,t_1).\]
\end{definition}

\begin{lemma}\label{lem:D^2_1}
The formula of the generating function $D_1^2$ is given by $$D^2_1=\frac{1+t_2}{1-t_1t_2}.$$
\end{lemma}

\begin{proof}
We classify all paths into two classes: the paths which have no reflection, and those with first reflection occurring at the $2^{\rm{nd}}$ plate. Then we have
\begin{equation}\label{D^2_1}
    D_{1}^{2} = 1+t_{2} D_{1}^{2}{'}.
\end{equation}
and, taking prime on both sides,
\begin{equation}\label{D^2_1'}
    D^2_1{'}=1+t_1D^2_1.
\end{equation}
From (\ref{D^2_1}) and (\ref{D^2_1'}), we have
\[D^2_1=1+t_2(1+t_1D^2_1)\]
and the conclusion follows.
\end{proof}

\newpage
\begin{remark}
Note that $D^2_1=({1+t_2})/({1-t_1t_2})=(1+t_2)\sum_{i\ge 0}t_{1}^{i}t_{2}^{i}$. It implies that each of the vectors $(0,0),(0,1),(1,1),(1,2),(2,2),(2,3),\cdots$ corresponds to exact one path respectively and the others have no corresponding path. Figure \ref{fig:D^2_1} justifies this result.
\end{remark}

\begin{figure}[htp]
    \centering
    \includegraphics[width=10cm]{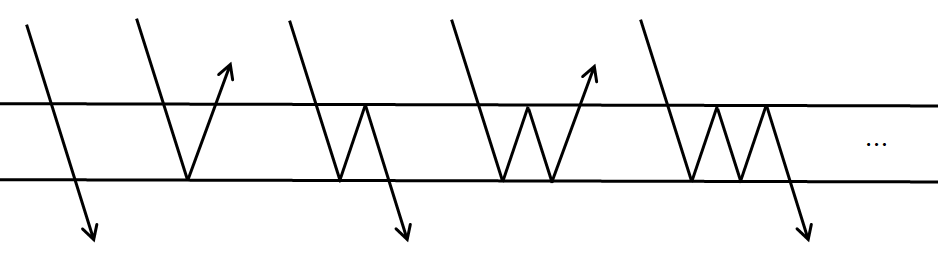}
    \caption{Paths within $2$ glass plates}
    \label{fig:D^2_1}
\end{figure}

The technique in proof of \ref{lem:D^2_1} can be applied to compute $D^n_1$ for $n\ge 3$. However, this process can be very complicated for large $n$. Next, we will provide another recursion for later discussion instead.

\begin{lemma}\label{lem:D^n_i}
For positive integers $n\ge 2$, $i\le n$, the generating functions $D^n_i$ satisfies
\[D^n_i=D^n_1-\sum_{j=2}^{i}t_j-\sum_{1\le a<b\le i}t_b t_a D^n_a\]
\end{lemma}

\begin{proof}
The set of all paths starting at the $i^{\rm{th}}$ plate, is equal to the set of all paths starting at the first plate but without those paths with first reflection occurring at $j^{\rm{th}}$ plate for some $j<i$. Moreover, the set of paths with first reflection occurring at $j^{\rm{th}}$ plate for some $j<i$ consists of paths with only one reflection and those have at least two reflections. So we have
$D^n_i=D^n_1-(\sum_{j=2}^{i}t_j+\sum_{1\le a<b\le i}t_b t_a D^n_a)$
\end{proof}

\begin{theorem}\label{thm:D^n_1}
For positive integer $n\ge 2$, the generating functions $D^{n}_1$ satisfy
\[D^{n}_{1}=\frac{nu(D^{n-1}_{1})+t_{n}D_{1}^{n\prime}}{de (D^{n-1}_1)}\]
where $nu(D^{n-1}_1)$ is the numerator of $D^{n-1}_1$, and $de(D^{n-1}_1)$ is the denominator of $D^{n-1}_1$.
\end{theorem}

\begin{proof}
Iterating lemma \ref{lem:D^n_i}, for $n\ge 2$, we can see that each term of $nu(D^n_1)$ except for $1$ has odd degree, and each term of $de(D^n_1)$ has even degree.

Now we classify paths in $D^n_1$ into two classes: paths containing a reflection which occurs at the $n^{\rm{th}}$ plate, and paths whose reflections do not occur at the $n^{\rm{th}}$ plate. A path in the first class must be a path with even reflections in $D^{n-1}_1$, connecting with a bottom up path starting at the $n^{\rm{th}}$ plate. So the generating function of this class is
\[e(D^{n-1}_1)t_n D^{n\prime}_1\]
where $e(D^{n-1}_1)$ is the even degree terms of $D^{n-1}_1$. The generating function of the second class is $D^{n-1}_1$. Then we get
\begin{equation}\nonumber
    \begin{aligned}
    D^{n}_1\times de(D^{n-1}_1)
    &=[D^{n-1}_1 +e(D^{n-1}_1)t_n D^{n\prime}_1]\times de(D^{n-1}_1)\\
    &=D^{n-1}_1\times de(D^{n-1}_1)+e(D^{n-1}_1)t_n D^{n\prime}_1\times de(D^{n-1}_1)\\
    &=nu(D^{n-1}_1)+t_n D^{n\prime}_1 [e(D^{n-1}_1)\times de(D^{n-1}_1)].
    \end{aligned}
    \end{equation}
Since $D^{n-1}_1\times de(D^{n-1}_1)=nu(D^{n-1}_1)$ and each term of $de(D^{n-1}_1)$ has even degree, we know that the sum of even degree terms in $nu(D^{n-1}_1)$ is $e(D^{n-1}_1)\times de(D^{n-1}_1)$. The only term of $nu(D^{n-1}_1)$ of even degree is $1$, and then we have $e(D^{n-1}_1)\times de(D^{n-1}_1)=1$.
\end{proof}

\begin{example}\label{exa:D^3_1}
The above theorem gives a recursion to compute $D^{n}_1$. For example, by \ref{lem:D^2_1}, we know $D^2_1=(1+t_2)/(1-t_1t_2)$, then applying theorem \ref{thm:D^n_1},
\begin{equation}\label{D^3_1}
 D^3_1=\frac{1+t_2+t_3 D^{3\prime}_1}{1-t_1 t_2} 
\end{equation}
and taking prime on both sides,
\begin{equation}\label{D^3_1'}
    D^{3\prime}_1=\frac{1+t_2+t_1 D^{3}_1}{1-t_3 t_2}.
\end{equation}
From (\ref{D^3_1}) and (\ref{D^3_1'}), we have
\[D^3_1=\frac{1+t_2+t_3 [\frac{1+t_2+t_1 D^{3}_1}{1-t_3 t_2}]}{1-t_1 t_2}.\]
Hence
\[D^3_1=\frac{1+t_2+t_3-t_2^2 t_3}{1-t_1 t_2 -t_2 t_3 -t_3 t_1 +t_1 t_2^2 t_3}.\]
\end{example}

\begin{theorem}
For $v\in\mathbb{Z}^n$, we define $m^n_v=N(v)$. Let $v\in\mathbb{Z}^n$ with $v_1 v_n\neq 0$, $n\ge 2$. If $v\cdot\vec{1}$ is odd, then
    \[m^n_v=m^n_{v-e_1 -e_2}+m^n_{v-e_n}+\sum_{\substack{u\in\mathbb{Z}^{n-2}\\ u\cdot\vec{1}\  \rm{is}\  \rm{odd}}}m^{n-2}_u m^n_{v-\overline{u}-e_1}+\sum_{\substack{0\neq u\in\mathbb{Z}^{n-2}\\ u\cdot\vec{1}\  \rm{is}\  \rm{even}}}m^{n-2}_u m^n_{v-\overline{u}-e_n}\]
and, if $v\cdot\vec{1}$ is even, then
    \[m^n_v=m^n_{v-e_n -e_{n-1}}+m^n_{v-e_1}+\sum_{\substack{u\in\mathbb{Z}^{n-2}\\ u\cdot\vec{1}\  \rm{is}\  \rm{odd}}}m^{n-2}_u m^n_{v-r(\overline{u})-e_n}+\sum_{\substack{0\neq u\in\mathbb{Z}^{n-2}\\ u\cdot\vec{1}\  \rm{is}\  \rm{even}}}m^{n-2}_u m^n_{v-r(\overline{u})-e_1}\]
where $\overline{u}=(0,u,0)$, $\{e_i\}$ is the standard basis, and $r(u)$ is the reverse of $u$.
\end{theorem}

\begin{proof}
Suppose $v\cdot\vec{1}$ is odd. First, we classify all paths within $n$ glass plates into four classes (see figure \ref{fig:recursion of N(v)}): the paths whose first reflection occur at $2^{\rm{nd}}$ plate; the paths whose first reflection occur at $n^{\rm{th}}$ plate; the paths, with last reflection occurring at $1^{\rm{st}}$ plate and other reflections occurring between $2^{\rm{th}}$ and $(n-1)^{\rm{th}}$ plates, connecting with a path starting at $1^{\rm{th}}$ plate, and the paths not belonging to previous three classes. This partition gives us
\begin{equation}\label{eq:recursion of m}
m^n_v=m^n_{v-e_1-e_2}+m^n_{r(v-e_n)}+\sum_{\substack{u\in\mathbb{Z}^{n-2}\\ u\cdot\vec{1}\  \rm{is}\  \rm{odd}}}m^{n-2}_u m^n_{v-\overline{u}-e_1}+\sum_{\substack{0\neq u\in\mathbb{Z}^{n-2}\\ u\cdot\vec{1}\  \rm{is}\  \rm{even}}}m^{n-2}_u m^n_{r(v-\overline{u}-e_n)}.
\end{equation}

Note that if $v'\cdot\vec{1}$ is even, then ``paths corresponding to $v'$'' and ``paths corresponding to $r(v')$'' have one-to-one correspondence, which implies $m^n_{v'}=m^n_{r(v')}$. Therefore, we have $m^n_{r(v-e_n)}=m^n_{v-e_n}$ and $m^n_{r(v-\overline{u}-e_n)}=m^n_{v-\overline{u}-e_n}$, and together with equation (\ref{eq:recursion of m}) the conclusion follows. For $v\cdot\vec{1}$ is even, the proof is similar.
\end{proof}

\begin{figure}[htp]
    \centering
    \includegraphics[width=12cm]{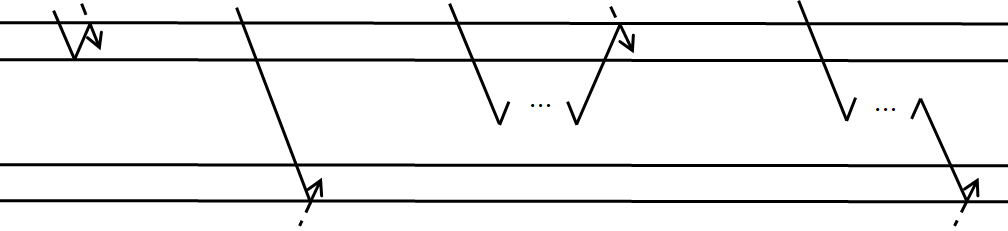}
    \caption{Paths within $n$ glass plates with $v_1 v_n\neq 0$}
    \label{fig:recursion of N(v)}
\end{figure}

\begin{remark}
If $v_1 v_n=0$, say $v_n=0$, then we reduce the problem to case $v\in\mathbb{Z}^{n-1}$.
\end{remark}
%-------------------------------------------------------------------------------------------------------------------------------------------------------------------------------------------------------------------------

\section{A closed form of $N(v)$ within $3$ glass plates}
From now on, we focus on the case $n=3$. Recalling example \ref{exa:D^3_1}, we have
\[D^3_1=\frac{1+t_2+t_3-t_2^2 t_3}{1-t_1 t_2 -t_2 t_3 -t_3 t_1 +t_1 t_2^2 t_3}.\]
Note that the coefficient of the term $t_1^x t_2^y t_3^z$ in $D^3_1$ is $N(x,y,z)$. To find a closed form, we compute the Taylor series of $D^3_1$, and put the coefficient of the term $t_1^n t_2^i t_3^j$ into matrices.
\begin{definition}
We define the matrix $M_n$ as the following: the $(i,j)$ entry of $M_n$ is $m^n_{i,j}=N(n,i,j)$, and we require that the indices $i$ and $j$ start from $0$. Notice that we do not determine the size of $M_n$ at this moment.
\end{definition}

\begin{remark}
It is easy to observe that matrix $M_n$ has $2n+3$ nonzero sequences which are parallel to the main diagonal (See figure 4 to figure 7).
\end{remark}

\begin{figure}[htp]
\begin{center}
\begin{minipage}[t]{0.48\textwidth}
\centering
\includegraphics[width=4cm]{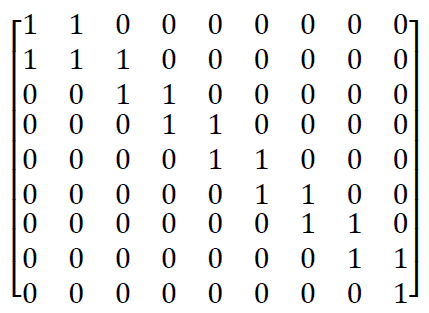}
\caption{Matrix $M_0$}
\end{minipage}
\begin{minipage}[t]{0.48\textwidth}
\centering
\includegraphics[width=4cm]{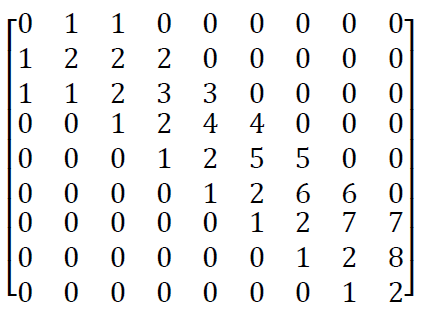}
\caption{Matrix $M_1$}
\end{minipage}
\end{center}
\end{figure}

\begin{figure}[htp]
\begin{center}
\begin{minipage}[t]{0.48\textwidth}
\centering
\includegraphics[width=6cm]{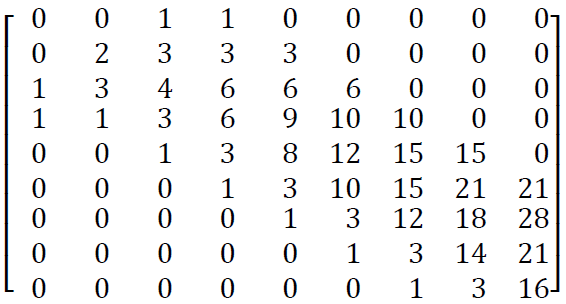}
\caption{Matrix $M_2$}
\end{minipage}
\begin{minipage}[t]{0.48\textwidth}
\centering
\includegraphics[width=6cm]{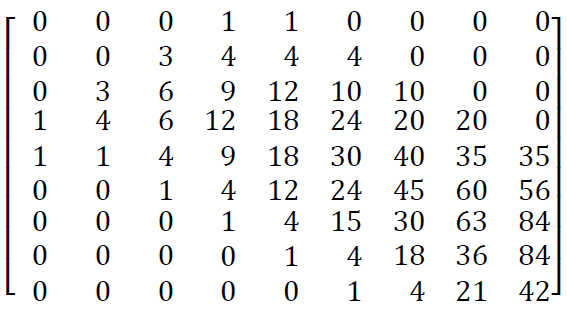}
\caption{Matrix $M_3$}
\end{minipage}
\end{center}
\end{figure}

\begin{definition}
In matrix $M_n$, \emph{the} $i^{\rm{th}}$ \emph{sequence} $S^n_i$ is the $i^{\rm{th}}$ nonzero sequence from the bottom of $M_n$ which is parallel to the diagonal. Moreover, we require the sequence starting from the first nonzero term. Figure $8$ marks the $1^{\rm{st}}$ to the $9^{\rm{th}}$ sequences of $M_3$.
\end{definition}

\begin{definition}
Given a sequence $S=\{a_n\}_{n\ge 0}$, we call itself the \emph{difference sequence of $S$ of order $0$}, and by recurrence, we define the \emph{difference sequence of $S$ of order $i$} to be $d(S,i)=\{c_{n+1}-c_n\}_{n\ge 0}$ where $\{c_n\}_{n\ge 0}$ is the difference sequence of $S$ of order $i-1$. Moreover, a sequence $\{a_n\}_{n\ge 0}$ is called an \emph{arithmetic sequence of order} $0$ if it is a constant sequence. A sequence $\{a_n\}_{n\ge 0}$ is called an \emph{arithmetic sequence of order} $k$ if the sequence $\{a_{n+1}-a_n\}_{n\ge 0}$ is an arithmetic sequence of order $k-1$. Figure $9$ gives an arithmetic sequence of order $3$. Finally, the leading term of a sequence $S$ is denoted by $LT(S)$.
\end{definition}

\begin{figure}[htp]
\begin{center}
\begin{minipage}[t]{0.48\textwidth}
\centering
\includegraphics[width=6cm]{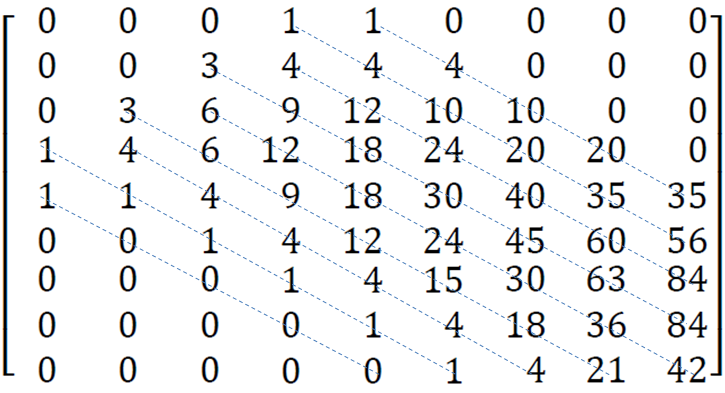}
\caption{The $i^{\rm{th}}$ sequences of $M_n$}
\end{minipage}
\begin{minipage}[t]{0.48\textwidth}
\centering
\includegraphics[width=6cm]{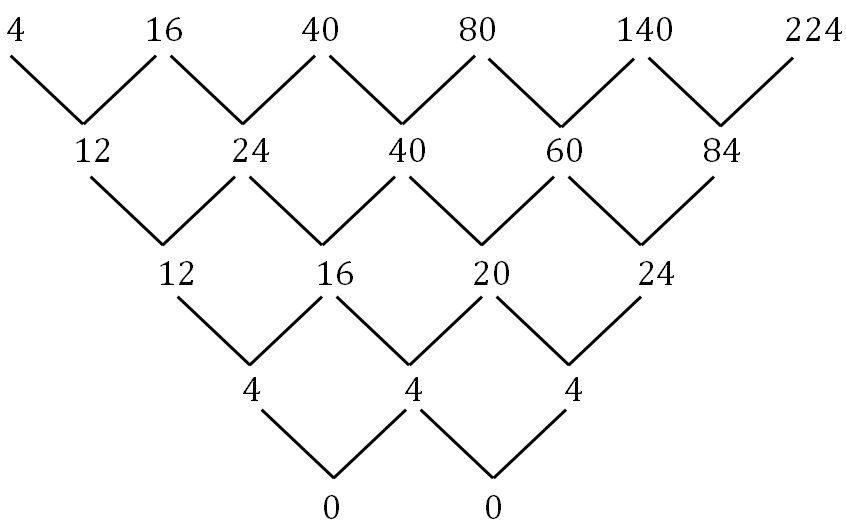}
\caption{Arithmetic sequence of higher order}
\end{minipage}
\end{center}
\end{figure}

Observing figure $4$ to figure $7$, we can see that $LT(S^n_1)=1$ and the other terms of $S^n_1$ are $0$. We also observe that 
\begin{equation*}
    LT(S^n_i)=
    \begin{cases}
    C(n+1,\frac{i-1}{2}) & \textrm{if $i$ is odd}\\
    C(n,\frac{i}{2}-1) & \textrm{if $i$ is even}.
    \end{cases}
\end{equation*}
Finally, for integers $n\ge k\ge 0$, $S^n_{2k+2}$ and $S^n_{2k+3}$ are arithmetic sequences of order $k$. These properties will derive a closed form for any entry of $M_n$. We will prove these observations in up-coming theorems.

\begin{lemma}[Tick Lemma]%勾勾性質
\label{lem:tick}
The elements $m^n_{i,j}$ of $M_n$ satisfy the recursion
\[m^n_{i,j}=m^{n-1}_{i-1,j}+\sum_{k\ge 0}m^{n-1}_{i-k,j-1-k}.\]
\end{lemma}

\begin{proof}
We classify all paths into two classes: paths with first reflection occurring at $2^{\rm{nd}}$ plate, and paths with first reflection occurring at $3^{\rm{rd}}$ plate. Hence we have
\[N(n,i,j)=N(n-1,i-1,j)+N'(n,i,j-1)\]
where $N'(c_1,c_2,c_3)$ is the number of bottom up paths which starting at the $3^{\textit{rd}}$ plate, and the $i^{\text{th}}$ plate has exactly $c_i$ reflections. Inductively, we get
\begin{equation}\nonumber
\begin{aligned}
N(n,i,j)&=N(n-1,i-1,j)+N'(n,i,j-1)\\
&=N(n-1,i-1,j)+N(n-1,i,j-1)+N'(n,i-1,j-2)\\
&=N(n-1,i-1,j)+N(n-1,i,j-1)+N(n-1,i-1,j-2)+N'(n,i-2,j-3)\\
&\hspace{3.25cm}\vdots\\
&=N(n-1,i-1,j)+\sum_{k\ge 0}N(n-1,i-k,j-1-k).
\end{aligned}
\end{equation}
\end{proof}

\begin{remark}%勾勾性質圖解
 By lemma \ref{lem:tick}, the $(3,4)$ entry in figure $11$ satisfies 
\[m^3_{3,4}=m^2_{2,4}+m^2_{3,3}+m^2_{2,2}+m^2_{1,1}+m^2_{0,0}.\]
which is the sum of some elements forming a tick in figure $10$. In general, $m^n_{i,j}$ is equal to the sum of a tick in $M_{n-1}$.    
\end{remark}

\begin{figure}[htp]
\begin{center}
\begin{minipage}[t]{0.48\textwidth}
\centering
\includegraphics[width=6cm]{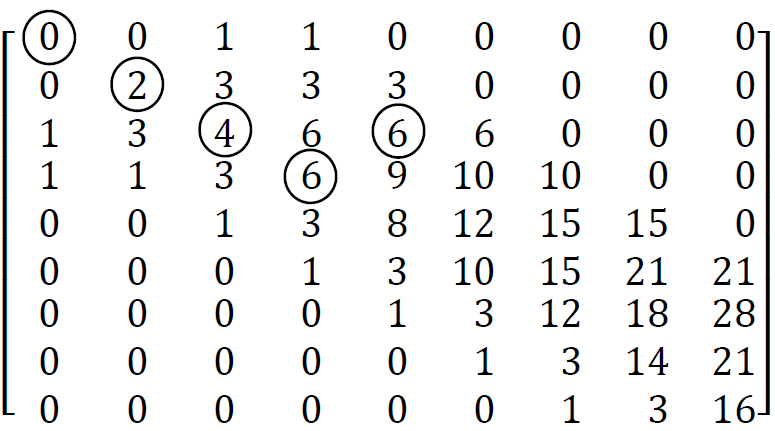}
\caption{Matrix $M_2$}
\end{minipage}
\begin{minipage}[t]{0.48\textwidth}
\centering
\includegraphics[width=6cm]{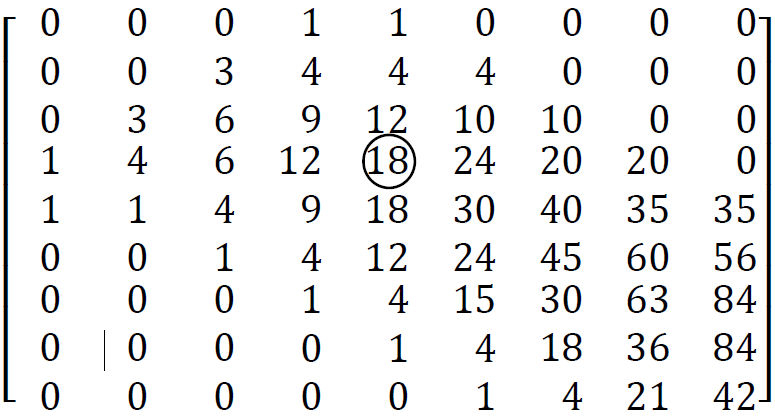}
\caption{Matrix $M_3$}
\end{minipage}
\end{center}
\end{figure}

\begin{theorem}%第i數列是階差數列
For nonnegative integer $n\ge 0$, $i\in\{1,2,\cdots,2n+3\}$, $S^n_i$ are the only $2n+3$ nonzero sequences parallel to diagonal in $M_n$. Moreover, each term of $S^n_1$ is $0$ except $LT(S^n_1)=1$, and for any $n\ge k\ge 0$, $S^n_{2k+2}$ and $S^n_{2k+3}$ are arithmetic sequences of order $k$.
\end{theorem}

\begin{proof}
We give a proof by induction on $n$. It is trivial to verify the case $n=0,1$. For $n\ge 2$, for any $n-1\ge k\ge 1$, $S^n_{2k+2}=\{m^n_{(n-1-k+i),(k+i)}\}_{i\ge 0}$, then lemma \ref{lem:tick} states that
\[m^n_{(n-1-k+i),(k+i)}=m^{n-1}_{(n-2+k+i),(k+i)}+\sum_{j\ge 0}m^{n-1}_{(n-1-k+i-j),(k+i-1-j)}.\]
By induction hypothesis, for any $i\ge 0$,  $\{m^{n-1}_{(n-1-k+i-j),(k+i-1-j)}\}_{j\ge 0}$ is an arithmetic sequence of order $k-1$, so $\{\sum_{j\ge 0}m^{n-1}_{(n-1-k+i-j),(k+i-1-j)}\}_{i\ge 0}$ is an arithmetic sequence of order $k$. Applying hypothesis again, $\{m^{n-1}_{(n-2+k+i),(k+i)}\}_{i\ge 0}$ is an arithmetic sequence of order $k$. Hence we know $S^n_{2k+2}$ is an arithmetic sequence of order $k$. A similar argument works for other cases.  
\end{proof}

\begin{theorem}%次對角有二項式係數
\label{thm:anti-diagonal}
For integers $n\ge0$, $2n+3\ge i\ge 1$, there holds
\begin{equation*}
    LT(S^n_i)=
    \begin{cases}
    C(n+1,\frac{i-1}{2}) & \textrm{if $i$ is odd}\\
    C(n,\frac{i}{2}-1) & \textrm{if $i$ is even}.
    \end{cases}
\end{equation*}
%In matrix $M_n$, the first two nonzero sequences are parallel to anti-diagonal consist of binomial coefficients $C(n,i)$ where $i\in\{1,2,\cdots,n\}$, and $C(n+1,i)$ where $i\in\{1,2,\cdots,n+1\}$ respectively. 
\end{theorem}

\begin{proof}
In matrix $M_n$, since $m^n_{i,j}$ is the number of paths corresponding to $(n,i,j)$, the entries $m^n_{a,b}$ in the first nonzero sequence parallel to the anti-diagonal must satisfy $a+b=n$, and the entries $m^n_{c,d}$ in the second sequence must satisfy $c+d=n+1$. Therefore, whenever $c+d=n+1$, $m^n_{c,d}=C(n+1,c)$ because it is the number of permutation of $c$ v-shaped paths whose reflection occurs at second plate and $d$ v-shaped paths whose reflection occurs at third plate. Figure $12$ shows all paths corresponding to $(3,2,2)$. Similarly, we have $m^n_{a,b}=C(n,a)$ for $a+b=n$.
\end{proof}

\begin{figure}[htp]
    \centering
    \includegraphics[width=8.7cm]{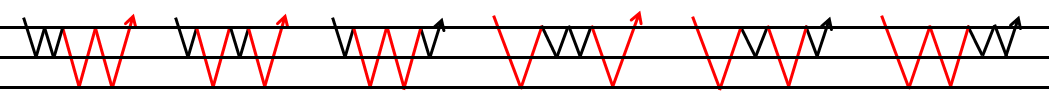}
    \caption{Paths corresponding to $(3,2,2)$}
\end{figure}

\newpage
To obtain a closed form of $m^n_{i,j}$, we need one more property about matrices $M_n$: figure $13$ shows $d(S^4_9,j)$ for $0\le j\le 3$. Note that the ratio of $LT(d(S^4_9,j))$ is $1:3:3:1$. Similarly, we have the ratio of $LT(d(S^5_{11},j))$ is $1:4:6:4:1$ (see figure $14$). In general, in matrix $M_n$, for any $i\ge 2$, ratio of the leading terms of difference sequences of $S^n_i$ is the binomial coefficients.

\begin{figure}[htp]
\begin{center}
\begin{minipage}[t]{0.48\textwidth}
\centering
\includegraphics[width=3.9cm]{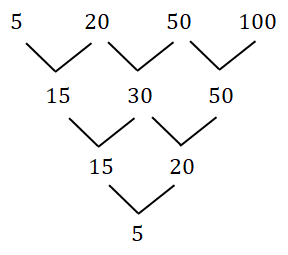}
\caption{The difference sequences of $S^4_9$}
\end{minipage}
\begin{minipage}[t]{0.48\textwidth}
\centering
\includegraphics[width=4.8cm]{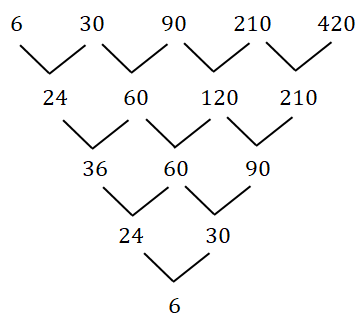}
\caption{The difference sequences of $S^5_{11}$}
\end{minipage}
\end{center}
\end{figure}

\begin{lemma}
\label{lem:LT}
In matrix $M_n$, the leading term of the difference sequence of the $i^{\rm{th}}$ sequence of order $j$ satisfies the recursion: 
\[LT(d(S^n_i,j))=LT(d(S^{n-1}_i,j))+LT(d(S^{n-1}_{i-2},j))+LT(d(S^{n-1}_{i-2},j-1)).\]
\end{lemma}

\begin{proof}
Let $\{a_j\}_{j\ge 0}$ and $\{b_j\}_{j\ge 0}$ be the $i^{\rm{th}}$ and the $(i-2)^{\rm{th}}$ sequence of $M_{n-1}$ respectively. By lemma \ref{lem:tick}, we see 
\[LT(d(S^n_i,j))=\sum_{k=0}^{j}(-1)^k C(j,k)a_{j-k}+\sum_{k=0}^{j-1}(-1)^k C(j-1,k)b_{j-k}.\]
Furthermore, it is also trivial to compute
\[LT(d(S^{n-1}_i,j))=\sum_{k=0}^{j}(-1)^k C(j,k)a_{j-k}\]
and
\begin{equation}\nonumber
\begin{aligned}
    LT(d(S^{n-1}_{i-2},j))+LT(d(S^{n-1}_{i-2},j-1))&=\sum_{k=0}^{j}(-1)^k C(j,k)b_{j-k}+\sum_{k=0}^{j-1}(-1)^k C(j-1,k)b_{(j-1)-k}\\
    &=\sum_{k=0}^{j-1}(-1)^k C(j-1,k)b_{j-k}
\end{aligned}
\end{equation}
which complete the proof.
\end{proof}

\newpage
\begin{theorem}%各階差數列首項比為二項式係數
\label{thm:LT}
In matrix $M_n$, for $i\ge 2$, $0\le j\le \lfloor{i/2}\rfloor-1$, there exists a constant $k_{n,i}$ such that
\begin{equation*}
  LT(d(S^n_i,j))=
    \begin{cases}
    k_{n,i}C(\frac{i}{2}-1,j) \textrm{ if $i$ is even}\\
    k_{n,i}C(\frac{i-1}{2}-1,j) \textrm{ if $i$ is odd.}
    \end{cases}  
\end{equation*}
\end{theorem}

\begin{proof}
We give a proof by induction on $n$. First, it is trivial to verify the case $n=0$. For $n\ge 1$, $i\ge 2$, we can apply lemma \ref{lem:LT} to get
\[LT(d(S^n_i,j))=LT(d(S^{n-1}_i,j))+LT(d(S^{n-1}_{i-2},j))+LT(d(S^{n-1}_{i-2},j-1)).\]
If  $i$ is even, by induction hypothesis, we have
\begin{equation}\nonumber
\begin{aligned}
    LT(d(S^n_i,j))&=LT(d(S^{n-1}_i,j))+LT(d(S^{n-1}_{i-2},j))+LT(d(S^{n-1}_{i-2},j-1))\\
    &=k_{n-1,i}C(\frac{i}{2}-1,j)+k_{n-1,i-2}C(\frac{i-2}{2}-1,j)+k_{n-1,i-2}C(\frac{i-2}{2}-1,j-1)\\
    &=k_{n-1,i}C(\frac{i}{2}-1,j)+k_{n-1,i-2}C(\frac{i-2}{2},j)\\
    &=[k_{n-1,i}+k_{n-1,i-2}]C(\frac{i}{2}-1,j).
\end{aligned}
\end{equation}
Similar, if $i$ is odd, we have
\begin{equation}\nonumber
\begin{aligned}
    LT(d(S^n_i,j))&=LT(d(S^{n-1}_i,j))+LT(d(S^{n-1}_{i-2},j))+LT(d(S^{n-1}_{i-2},j-1))\\
    &=k_{n-1,i}C(\frac{i-1}{2}-1,j)+k_{n-1,i-2}C(\frac{i-3}{2}-1,j)+k_{n-1,i-2}C(\frac{i-3}{2}-1,j-1)\\
    &=k_{n-1,i}C(\frac{i-1}{2}-1,j)+k_{n-1,i-2}C(\frac{i-3}{2},j)\\
    &=[k_{n-1,i}+k_{n-1,i-2}]C(\frac{i-1}{2}-1,j).
\end{aligned}
\end{equation}
\end{proof}

\begin{example}\label{exa:closed form}
The number of paths corresponding to $(7,5,7)$ is $N(7,5,7)=840$.
\end{example}

\begin{proof}
By definition, $N(7,5,7)=m^7_{5,7}$ and $m^7_{5,7}$ is the $3^{\rm{rd}}$ term of $S^7_{11}$ (see figure $15$). Theorem \ref{thm:anti-diagonal} says that $LT(S^7_{11})=C(8,3)=56$, and ratio of leading terms of difference sequences of $S^7_{11}$ is  $1:4:6:4:1$ by theorem \ref{thm:LT}. So we get a diagram of difference sequences in figure $16$ where the question mark is $m^7_{5,7}$. Finally, $S^7_{11}$ is an arithmetic sequence of order $4$, which implies $m^7_{5,7}=56+224\times 2+336=840$.
\end{proof}

\begin{figure}[htp]
\begin{center}
\begin{minipage}[t]{0.48\textwidth}
\centering
\includegraphics[width=6cm]{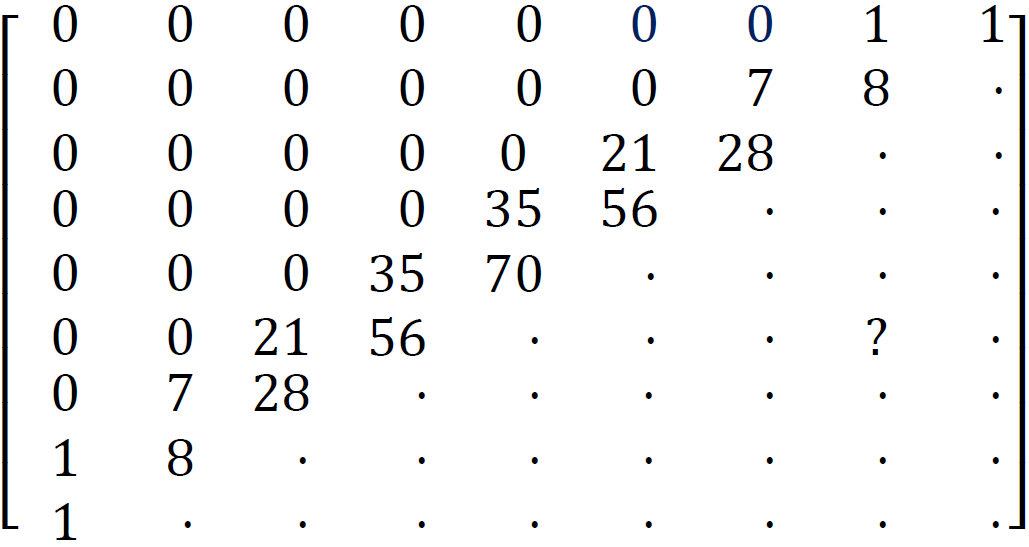}
\caption{The matrix $M_7$}
\end{minipage}
\begin{minipage}[t]{0.48\textwidth}
\centering
\includegraphics[width=4.2cm]{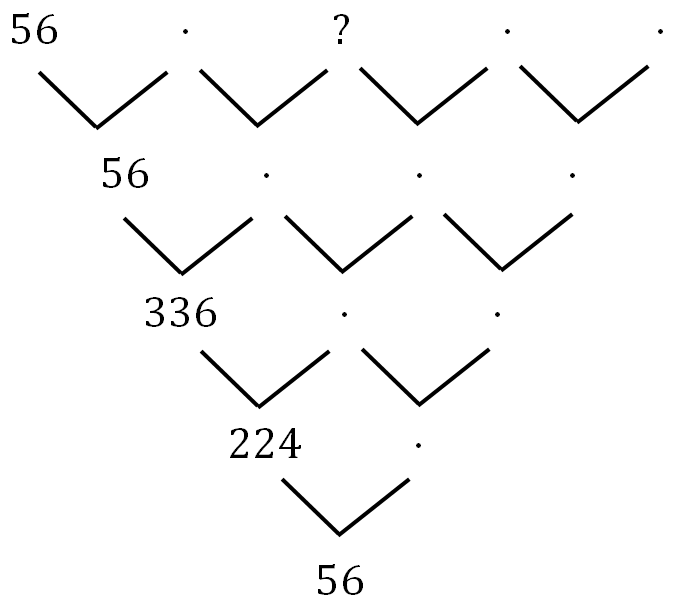}
\caption{Difference sequences of $S^7_{11}$}
\end{minipage}
\end{center}
\end{figure}

Generalizing the method in example \ref{exa:closed form}, we obtain the following closed form of $N(v)$:
\begin{theorem}%closed form of N(v)
\label{thm:closed form}
For integers $n,i,j\ge 0$, the number of paths corresponding to $(n,i,j)$ is
\[N(n,i,j)=m^n_{i,j}=C(k,\frac{k+j-i}{2})\times C(\lfloor{\frac{\ell-2}{2}}\rfloor+\lfloor{\frac{i+j-k}{2}}\rfloor,\lfloor\frac{\ell-2}{2}\rfloor)\]
where $k=n-1+\gcd(2,i+j+n+1)$, $\ell=k+j-i+\gcd(2,i+j+n)$, and $C(\alpha,-1)=1$ if $\alpha =-1$ and $C(\alpha,-1)=0$ if $\alpha\neq -1$.
\end{theorem}

\begin{proof}
In the matrix $M_n$, by theorem \ref{thm:anti-diagonal}, the numbers of the first two nonzero sequences parallel to the anti-diagonal are $C(n,t)$, $t\in\{0,1,\cdots,n\}$ and $C(n+1,s)$, $s\in\{0,1\cdots,n+1\}$ respectively. Suppose $m^n_{i,j}$ is in $S^n_{\ell}$, $\ell\ge 2$, then the leading term $LT(S^n_{\ell})=m^n_{x,y}$ satisfies
\[2 \mid [(x+y)-(i+j)].\]
Hence $LT(S^n_{\ell})=C(k,(k+j-i)/2)$ where
\[k=n-1+\gcd(2,i+j+n+1) \quad\textrm{and}\quad \ell=k+j-i+\gcd(2,i+j+n).\]
Finally, since $S^n_{\ell}$ is an arithmetic sequence of order $\lfloor(\ell-2)/2\rfloor$ and $m^n_{i,j}$ is the $\{\lfloor(i+j-k)/2\rfloor+1\}^{\rm{th}}$ term of $S^n_{\ell}$, we have
\begin{equation}\nonumber
\begin{aligned}
    m^n_{i,j}&=C(k,\frac{k+j-i}{2})\sum_{p\ge 0}C(\lfloor{\frac{i+j-k}{2}}\rfloor,p)\times C(\lfloor{\frac{\ell-2}{2}}\rfloor,\lfloor{\frac{\ell-2}{2}}\rfloor-p)\\
    &=C(k,\frac{k+j-i}{2})\times C(\lfloor{\frac{\ell-2}{2}}\rfloor+\lfloor{\frac{i+j-k}{2}}\rfloor,\lfloor\frac{\ell-2}{2}\rfloor).
\end{aligned}
\end{equation}
\end{proof}

\section{Existence of a path corresponding to a given vector}
Given a vector $(x,y,z)\in\mathbb{Z}^3$, there is a path corresponding to $(x,y,z) $ if and only if $m^x_{y,z}=N(x,y,z)\neq 0$. So we can use theorem \ref{thm:closed form} to determine whether there is a path corresponding to a given vector or not. However, in this section, we will give a better criterion.

\begin{definition}
For every alternating word $w$, $P_w$ denote the path corresponding to $w$. 
\end{definition}

\begin{definition}
For alternating words $w_1,w_2$, we say $P_{w_1}$ and $P_{w_2}$ are \emph{connectable} if $w_1 w_2$ is an alternating word. In this case, we denote the path corresponding to $w_1 w_2$ by $[P_{w_1},P_{w_2}]$.
\end{definition}

\begin{lemma}\label{lem:connected}
The following pairs of paths are connectable:
\[(P_{21},P_{32}),\ (P_{32},P_{31}),\ (P_{21},P_{21}),\ (P_{32},P_{32}),\ (P_{31},P_{31}),\ (P_{21},P_{31}).\]
\end{lemma}

\begin{proof}
$2132,3231,2121,3232,3131$, and $2131$ are alternating words. Figure $17$ gives the connected paths of these pairs in order. 
\end{proof}

\begin{figure}[htp]
    \centering
    \includegraphics[width=12cm]{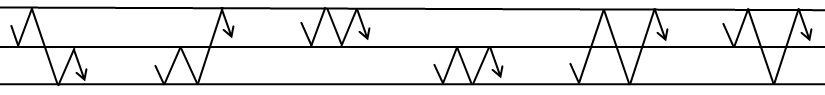}
    \caption{Connected paths of $6$ pairs of paths}
\end{figure}

\begin{theorem}
\label{thm:even case}
For $(a_1,a_2,a_3)\in\mathbb{Z}^3$ with $a_1,a_2,a_3\ge 0$ and $a_1+a_2+a_3$ being even, there is a path corresponding to $(a_1,a_2,a_3)$ if and only if
\begin{equation}
\begin{cases}
a_1+a_2\ge a_3\\
a_2+a_3\ge a_1\\
a_3+a_1\ge a_2.
\end{cases}
\label{cond:even}
\end{equation}
\end{theorem}

\begin{proof}
Suppose $(a_1,a_2,a_3)$ is a vector satisfies the above condition. Consider the equation
\[ x(1,1,0)+y(0,1,1)+z(1,0,1)=(a_1,a_2,a_3)\]
and its only solution
\[\begin{cases}
x=(a_1+a_2-a_3)/2\\
y=(-a_1+a_2+a_3)/2\\
z=(a_1-a_2+a_3)/2.
\end{cases}\]
The hypothesis implies that $x,y,z$ are nonnegative integers. By lemma \ref{lem:connected}, we can construct a path by connecting $x$ copies of $P_{21}$, $y$ copies of $P_{32}$ and $z$ copies of $P_{31}$.

Conversely, let $P$ be a path with $2k$ reflections and $v(P)=(a_1,a_2,a_3)$. Notice that $P$ can be written as the form
\[[P_{w_1},P_{w_2},\cdots,P_{w_k}].\]  
where $w_i$ is an alternating word of length $2$ over alphabet $\{1,2,3\}$. Since $v(P_{w_i})=(1,1,0)$ or $(0,1,1)$ or $(1,0,1)$, there are integers $x,y,z\ge 0$ such that
\[(a_1,a_2,a_3)=x(1,1,0)+y(0,1,1)+z(1,0,1).\]
Hence
\[\begin{cases}
a_1+a_2=2x+y+z\ge y+z=a_3\\
a_2+a_3=x+2y+z\ge x+z=a_1\\
a_1+a_3=x+y+2z\ge x+y=a_2.
\end{cases}\]
\end{proof}

\begin{corollary}\label{thm:odd case}
For $(a_1,a_2,a_3)\in\mathbb{Z}^3$ with $a_1,a_2,a_3\ge 0$ and $a_1+a_2+a_3$ being odd, there is a path corresponding to $(a_1,a_2,a_3)$ if and only if one of the followings holds
\begin{enumerate}
    \item $\begin{cases}
           a_1+a_2\ge a_3\\
           a_2+a_3\ge a_1\\
           a_3+a_1\ge a_2.
           \end{cases}$ 
    \item There exists a constant $\alpha\ge 0$ such that $(a_1,a_2,a_3)=(\alpha,\alpha+1,0)$.
\end{enumerate}
\end{corollary}

Recall definition \ref{def}, $b^n_m$ is the number of distinct vectors corresponding to some paths with exactly $m$ reflections. Using \ref{thm:even case} and \ref{thm:odd case}, we obtain closed forms of the statistic $b^3_m$.

\begin{theorem}
For integer $n\ge 1$, we have
\begin{enumerate}
    \item $b^3_{2n}=C(n+2,2)$.
    \item $b^3_{2n+1}=C(n+2,2)+1$.
\end{enumerate}
\end{theorem}

\begin{proof}
By theorem \ref{thm:even case}, vector $(a_1,a_2,a_3)$ is corresponding to a path with $2n$ reflections if and only if
\[\begin{cases}a_1+a_2\ge a_3\\
a_2+a_3\ge a_1\\
a_3+a_1\ge a_2
\end{cases}\]
which is equivalence to $a_i\le n$ for $i=1,2,3$. In this case, by the principle of inclusion and exclusion, the number of such vectors $(a_1,a_2,a_3)$ is  
\[C(2n+2,2)-3C(n+1,2)=(n+1)(2n+1)-3\frac{(n+1)n}{2}=C(n+2,2).\]
Similar argument works for $b^3_{2n+1}$.
\end{proof}

%----------------------------------------------------------------------------------------------------------
\medskip
%Sets the bibliography style to UNSRT and imports the 
%bibliography file "samples.bib".
\bibliographystyle{alpha}
\bibliography{references}

\begin{thebibliography}{GKZ16}

\bibitem[GKZ16]{gao2016pattern}
Alice~LL Gao, Sergey Kitaev, and Philip~B Zhang.
\newblock Pattern-avoiding alternating words.
\newblock {\em Discrete Applied Mathematics}, 207:56--66, 2016.

\bibitem[HJ79]{hoggatt1979reflections}
VE~Hoggatt~Jr.
\newblock Reflections across two and three glass plates.
\newblock {\em Fibonacci Quarterly}, 17:118--142, 1979.

\bibitem[JHJ73]{junge1973polynomials}
Bjarne Junge and VE~Hoggatt~Jr.
\newblock Polynomials arising from reflections across multiple-plates.
\newblock {\em Fibonacci Quarterly}, 11(3):285--91, 1973.

\end{thebibliography}
%----------------------------------------------------------------------------------------------------------
\end{document}